\documentclass{amsart}
\usepackage{amssymb,amscd,amsmath,mathabx,enumerate,verbatim,calc}
\usepackage{amsopn,amsthm,graphics,amsfonts}
\usepackage{mathrsfs}
\usepackage[autostyle]{csquotes}
\usepackage[dvips]{graphicx}
\usepackage[autostyle]{csquotes}
\usepackage[colorlinks=true,linkcolor=red,citecolor=blue]{hyperref}
\usepackage{mathrsfs}
\usepackage{tikz-cd}
\usetikzlibrary{matrix,arrows,backgrounds}
\usepackage{tikz}
\usetikzlibrary{matrix,arrows,decorations.pathmorphing}
\input xy
\xyoption{all}
\calclayout

\makeatletter
\newcommand{\xRrightarrow}[2][]{\ext@arrow 0359\Rrightarrowfill@{#1}{#2}}
\newcommand{\Rrightarrowfill@}{\arrowfill@\equiv\equiv\Rrightarrow}
\newcommand{\xLleftarrow}[2][]{\ext@arrow 3095\Lleftarrowfill@{#1}{#2}}
\newcommand{\Lleftarrowfill@}{\arrowfill@\Lleftarrow\equiv\equiv}
\newcommand{\xLleftRrightarrow}[2][]{\ext@arrow 3399\LleftRrightarrowfill@{#1}{#2}}
\newcommand{\LleftRrightarrowfill@}{\arrowfill@\Lleftarrow\equiv\Rrightarrow}
\makeatother

\newcommand{\rt}{\rightarrow}

\newcommand{\st}{\stackrel}

\newcommand{\CE}{\mathcal{E}}

\newcommand{\CI}{\mathcal{I} }

\newcommand{\CL}{\mathcal{L} }

\newcommand{\CP}{\mathcal{P} }

\newcommand{\Prod}{{\rm{Prod}}}

\newcommand{\im}{{\rm{Im}}}

\newcommand{\Add}{{\rm{Add}}}
\newcommand{\add}{{\rm{add}}}

\newcommand{\gen}{{\rm{gen}}}
\newcommand{\cogen}{{\rm{cogen}}}

\newcommand{\RNum}[1]{\uppercase\expandafter{\romannumeral #1\relax}}

\newcommand{\pd}{{\rm{pd}}}
\newcommand{\id}{{\rm{id}}}
\newcommand{\fd}{{\rm{fd}}}

\newcommand{\Gid}{{\rm{Gid}}}
\newcommand{\Gpd}{{\rm{Gpd}}}
\newcommand{\Gfd}{{\rm{Gfd}}}

\newcommand{\Tr}{{\rm{Tr}}}

\newcommand{\Coker}{{\rm{Coker}}}
\newcommand{\Ker}{{\rm{Ker}}}

\newcommand{\HH}{{\rm{H}}}

\newcommand{\Hop}{{\rm{op}}}

\newcommand{\Tor}{{\rm{Tor}}}
\newcommand{\Hom}{{\rm{Hom}}}
\newcommand{\Ext}{{\rm{Ext}}}
\newcommand{\End}{{\rm{End}}}

\newtheorem{theorem}{Theorem}[section]
\newtheorem{corollary}[theorem]{Corollary}
\newtheorem{lemma}[theorem]{Lemma}

\newtheorem{definition}[theorem]{Definition}
\newtheorem{example}[theorem]{Example}

\newtheorem{remark}[theorem]{Remark}

\theoremstyle{plain}

\theoremstyle{definition}

\numberwithin{equation}{section}

\begin{document}

\author[K. Divaani-Aazar, A. Mahin Fallah and M. Tousi]
{Kamran Divaani-Aazar, Ali Mahin Fallah and Massoud Tousi}

\title[Cotilting modules and Gorenstein ...]
{Cotilting modules and Gorenstein homological dimensions}

\address{K. Divaani-Aazar, Department of Mathematics, Faculty of Mathematical Sciences,
Alzahra University, Tehran, Iran.}
\email{kdivaani@ipm.ir}

\address{A. Mahin Fallah, School of Mathematics, Institute for Research in Fundamental Sciences (IPM),
P.O. Box: 19395-5746, Tehran, Iran.}
\email{amfallah@ipm.ir, ali.mahinfallah@gmail.com}

\address{M. Tousi, Department of Mathematics, Faculty of Mathematical Sciences, Shahid Beheshti University,
Tehran, Iran.}
\email{mtousi@ipm.ir}


\keywords{Auslander class; Bass class; coherent ring; co-resolving subcategory;  cotilting module; Gorenstein
flat module; Gorenstein injective module; Gorenstein projective module; resolving subcategory; tilting module;
Wakamatsu tilting module.\\
The research of the second author is supported by a grant from IPM (No.1403160017).}

\begin{abstract}  For a dualizing module $D$ over a commutative Noetherian ring $R$ with identity, it is known
that its Auslander class $\mathscr{A}_D\left(R\right)$ (respectively, Bass class $\mathscr{B}_D\left(R\right)$)
is characterized as those $R$-modules with finite Gorenstein flat dimension (respectively, finite Gorenstein
injective dimension). We establish an analogue of this result in the context of cotilting modules over general
Noetherian rings.
\end{abstract}

\maketitle

\tableofcontents

\section{Introduction}

Semidualizing modules serve as a common generalization of dualizing modules and projective modules of rank one
over commutative Noetherian rings with identity. Let $C$ be a semidualizing module over a commutative Noetherian
ring $R$ with identity. The two classes of $R$-modules $\mathscr{A}_C(R)$ and $\mathscr{B}_C(R)$, known as the
Auslander and Bass classes, correspond to $C$. They are intimately related through an equivalence of categories
\begin{displaymath}
\xymatrix{\mathscr{A}_C\left(R\right) \ar@<0.7ex>[rrr]^-{C\otimes_R-} &
{} & {} & \mathscr{B}_C\left(R\right),  \ar@<0.7ex>[lll]^-{\Hom_R\left(C,-\right)}}
\end{displaymath}
known as Foxby equivalence.

Recall that dualizing modules are precisely the semidualizing modules that have finite injective dimension.
When $C$ is a dualizing $R$-module, there is a well-known characterization of the classes
$\mathscr{A}_C\left(R\right)$ and $\mathscr{B}_C\left(R\right)$ via Gorenstein homological dimensions. Specifically,
in this case:
\begin{itemize}
\item[(i)] $\mathscr{A}_C\left(R\right)$ consists precisely of all $R$-modules with finite Gorenstein flat dimension.
\item[(ii)] $\mathscr{B}_C\left(R\right)$ consists precisely of all $R$-modules with finite Gorenstein injective
dimension.
\end{itemize}
(See e.g. \cite{EJ}.) It is natural to expect an analogous result for general Noetherian (not necessarily commutative)
rings.
While some results exist in this direction under additional assumptions, including the perfection of the
underlying ring (\cite{EJL}, \cite{BGO}), our aim in this paper is to relax this strong assumption.

From now on, $R$ is an associative ring with identity. In 2007, Holm and White \cite{HW} defined semidualizing
bimodules over associative rings with identity.
For a semidualizing bimodule $_ST_R$, they defined the Auslander and Bass classes $\mathscr{A}_T\left(S\right)$
and $\mathscr{B}_T\left(R\right)$. An enhanced theory of Wakamatsu tilting modules exists over an associative
ring $R$ with identity. Surprisingly, Wakamatsu tilting modules can be classified as semidualizing bimodules;
see Lemma \ref{2.2}. In addition, parallel to the theory of dualizing modules, a rich theory of cotilting
modules exists within the representation theory of algebras.

Tilting theory plays a remarkable role in the representation theory of algebras. The classical tilting modules
were introduced in the context of finitely generated modules over finite-dimensional algebras by Brenner and
Butler \cite{BB} and Happel and Ringel \cite{HR}. Miyashita \cite{MI} extended this notion to encompass
finitely presented modules of finite projective dimension over arbitrary rings. In his seminal work, Wakamatsu
\cite{W} further generalized the concept of tilting modules by allowing for the possibility of infinite projective
dimension. These modified tilting modules are now commonly referred to as \enquote{Wakamatsu tilting modules},
officially adopting the terminology established in \cite{GRS}.

Cotilting modules were originally introduced as an analogous to dualizing modules over general Noetherian rings.
There are several definitions for cotilting modules. In this paper, we consider the definition given by Miyashita
in 1986. Let $C_R$ be a finitely generated $R$-module and $S=\End(C_R)$. Assume that the ring $R$ is right Noetherian
and the ring $S$ is left Noetherian. Miyashita defined $C_R$ to be cotilting if it is a Wakamatsu tilting
$R$-module and both $\id(C_R)$ and $\id(_SC)$ are finite.

We derive the following results to provide a precise description of the classes $\mathscr{A}_T\left(S\right)$ and
$\mathscr{B}_T\left(R\right)$:

\begin{theorem}\label{1.2} Let $C_R$ be a tensorly faithful cotilting module, and let $S=\End(C_R)$. Then, for
an $S$-module $N_S$, the following are equivalent:
\begin{itemize}
\item[(i)] $N_S\in \mathscr{A}_C\left(S\right)$.
\item[(ii)] the Gorenstein projective dimension of $N_S$ is finite.
\item[(iii)]  the Gorenstein flat dimension of $N_S$ is finite.
\end{itemize}
\end{theorem}

The definitions of Gorenstein homological dimensions will be recalled later in Definitions \ref{3.2} and \ref{4.2}.

\begin{theorem}\label{1.3} Let $C_R$ be a tensorly faithful cotilting module, and let $S=\End(C_R)$. Then, for an
$R$-module $M_R$, the following are equivalent:
\begin{itemize}
\item[(i)] $M_R\in \mathscr{B}_C\left(R\right)$.
\item[(ii)] the Gorenstein injective dimension of $M_R$ is finite.
\end{itemize}
\end{theorem}

Section 2 offers an overview of the background on Wakamatsu tilting modules. We prove Theorem \ref{1.2} in Section 3
(cf. Theorem \ref{3.5}) and Theorem \ref{1.3} in Section 4 (cf. Theorem \ref{4.4}).

\section{Preliminaries}

In this paper, we consider associative rings with identity, and all modules are assumed to be unitary. We use
the notation $M_R$ (respectively, $_RM$) to denote a right (respectively, left) $R$-module. The category of
all right $R$-modules is denoted by $\text{Mod-}R$. Let $M_R$ be an $R$-module. We denote by $\add(M)$ the
class of right $R$-modules which are isomorphic to a direct summand of a direct
sum of finitely many copies of $M$.

All subcategories that are considered throughout are full and closed under finite direct sums, direct summands,
and isomorphisms. A subcategory $\mathcal{C}$ of $\text{Mod-}R$ is said to be {\it resolving} if it is closed
under extensions and kernels of epimorphisms, and it contains all the projective modules. Similarly, a
subcategory $\mathcal{C}$ of $\text{Mod-}R$ is called {\it co-resolving} if it is closed under extensions and
cokernels of monomorphisms, and it contains all the injective modules.

For an $R$-module $T_R$, the symbol $\gen^*(T_R)$ stands for the class of $R$-modules $N_R$ for which there exists
an exact sequence of the form $$\cdots\st{f_2}\rt T_1 \st{f_1}\rt T_0 \st{f_0}\rt
N \rt 0$$ with each $T_i\in \add(T)$ and $\Ext^1_R(T,\Ker f_i)=0$ for all $i\geqslant 0$. Dually, the class
$\cogen^*(T_R)$ is consisting of $R$-modules $N_R$ for which there exists an exact sequence of the form $$0\rt N
\st{f^{-1}}\rt T^0 \st{f^0}\rt T^1
\st{f^1}\rt \cdots$$ with each $T^i\in \add(T)$ and $\Ext^1_R(\Coker f^i,T)=0$
for all $i\geqslant -1$.

For the remainder, $\pd(-)$, $\fd(-)$, and $\id(-)$ will denote the projective, flat, and injective dimensions of
modules, respectively.

\begin{definition}\label{2.1}  Let $T_R$ be an $R$-module and $n\in \mathbb{N}_0$. We say that $T_R$ is
$n$-{\it tilting} if
\begin{itemize}
\item[(i)]  $T_R\in \gen^*(R)$ and $\pd(T_R)\leq n$,
\item[(ii)] $\Ext^i_{R}(T,T)=0$ for all $i>0$, and
\item[(iii)] There is an exact sequence $$0\rt R\rt T_0 \rt T_1\rt \cdots \rt T_n\rt 0,$$ where $T_i\in \add(T)$ for all
$0\leq i\leq n$.
\end{itemize}
\end{definition}

\begin{definition}\label{2.1} (See \cite[Section 3]{W}.) Let $T_R$ be an $R$-module. We say that $T_R$ is a
{\it Wakamatsu tilting} module if the following conditions hold:
\begin{itemize}
\item[(i)] $T_R\in \gen^*(R)$,
\item[(ii)] $\Ext^i_{R}(T,T)=0$ for all $i>0$,
\item[(ii)] $R_R \in \cogen^* (T_R)$.
\end{itemize}
\end{definition}

It is easy to verify that every tilting module is also a Wakamatsu tilting module. The notion of Wakamatsu tilting left
modules is defined similarly. By \cite[Corollary 3.2]{W}, we have the following characterization of the Wakamatsu tilting
modules.

\begin{lemma}\label{2.2} For a bimodule $_ST_R$, the following are equivalent:
\begin{itemize}
\item[(i)] $T_R$ is a Wakamatsu tilting module with $S\cong \End(T_R)^{\Hop};$
\item[(ii)] $_ST$ is a Wakamatsu tilting module with $R\cong \End(_ST)^{\Hop};$
\item[(iii)] One has
\begin{itemize}
\item[(1)]  $T_R\in \gen^*(R)$  and $~_ST\in \gen^*(S)$.
\item[(2)]  $R\cong \End(_ST)^{\Hop}$ and $S\cong \End(T_R)^{\Hop}$.
\item[(3)]  $\Ext^i_{R}(T,T)=0$  and $\Ext^i_{S}(T,T)=0$ for all $i>0$.
\end{itemize}
\end{itemize}
\end{lemma}

Note that by Lemma \ref{2.2}, a Wakamatsu tilting module is a semidualizing bimodule in the sense of Holm and
White \cite{HW}.

\begin{definition}\label{2.3} Let $\mathscr{C}$ be a class of right $R$-modules. A homomorphism $f: C\rt M$
with $C\in \mathscr{C}$ is a $\mathscr{C}$-{\it precover} if for any homomorphism $g: C_0\rt M$  with
$C_0\in \mathscr{C}$, there exists a homomorphism satisfying $h:C_0 \rt C$ such that $g=fh$. A
$\mathscr{C}$-precover ${f}:C\rt M$ is called a $\mathscr{C}$-{\it cover} if every endomorphism
$\varphi:C\rt C$ with $f=f\varphi$ is an automorphism. The class $\mathscr{C}$ is called {\it
(pre)covering} if every right $R$-module has a $\mathscr{C}$-(pre)cover.
\end{definition}

Dually, ${\mathcal{C}}$-(pre)envelope and (pre)enveloping classes were defined.

Now, we recall the definitions of the Auslander and Bass classes. Our definitions here deviate slightly from those
in \cite{HW}, but this difference is merely a side-switching. Therefore, we can still refer to \cite{HW} in Lemmas
\ref{2.4} and \ref{2.9}, as well as in Theorems \ref{3.5} and \ref{4.4}.

\begin{definition}\label{2.3a} Let $_ST_R$ be a Wakamatsu tilting module.
\begin{itemize}
\item[(i)] The {\it Auslander} class $\mathscr{A}_T\left(S\right)$ with respect to $T$ consists of all
$S$-modules $N_S$ satisfying $\Tor_{i\geqslant1}^S(N,T)=0$, $\Ext^{i\geqslant1}_R(T,N\otimes_ST)=0$
and the natural map $$\theta_N^T: N\rt \Hom_R(T,N\otimes_ST)$$ is an isomorphism.
\item[(ii)] The {\it Bass} class $\mathscr{B}_T\left(R\right)$ with respect to $T$ consists of all
$R$-modules $M_R$ satisfying $\Ext^{i\geqslant1}_R(T,M)=0$, $\Tor^S_{i\geqslant1}(\Hom_R(T,M),T)=0$
and the natural map $$\nu_M^T:\Hom_R(T,M)\otimes_ST\rt M$$ is an isomorphism.
\end{itemize}
\end{definition}

The following lemma summarizes some of the well-known properties of the Auslander and Bass classes.

\begin{lemma}\label{2.4} Let $_ST_R$ be a Wakamatsu tilting module. Then
\begin{itemize}

\item[(1)] The class $\mathscr{A}_T\left(S\right)$ is resolving. Furthermore, the class $\mathscr{A}_T\left(S\right)$
contains all flat right $S$-modules.
\item[(2)]  The class $\mathscr{B}_T\left(R\right)$ is co-resolving. In particular, the class $\mathscr{B}_T\left(R\right)$
contains all injective right $R$-modules.
\item[(3)] There is an equivalence of categories
\begin{displaymath}
\xymatrix{\mathscr{A}_T\left(S\right) \ar@<0.7ex>[rrr]^-{-\otimes_ST} &
{} & {} & \mathscr{B}_T\left(R\right).  \ar@<0.7ex>[lll]^-{\Hom_R\left(T,-\right)}}
\end{displaymath}
\item[(4)] If $N_S , N'_S\in \mathscr{A}_T\left(S\right)$, then $\Ext^i_S(N,N')\cong \Ext^i_R(N\otimes_ST,N'\otimes_ST)$
for all $i\geqslant0$.
\item[(5)]  If $M_R , M'_R\in \mathscr{B}_T\left(R\right)$, then $\Ext^i_R(M,M')\cong  \Ext^i_S(\Hom_R(T,M),\Hom_R(T,M'))$
for all $i\geqslant0$.
\end{itemize}
\end{lemma}

\begin{proof} For (1) and (2), see \cite[Theorem 6.2 and Lemma 4.1]{HW}.

For (3), (4), and (5), see, respectively, Proposition 4.1, Theorem 6.4(1) and Theorem 6.4(2) in \cite{HW}.
\end{proof}

\begin{lemma}\label{2.6} Let $_ST_R$ be a bimodule.
\begin{itemize}
\item[(i)] If $R$ is right Noetherian, then $\id((F\otimes_ST)_R)\leq \id(T_R)$ for every flat $S$-module $F_S$.
\item[(ii)] If $S$ is left coherent, then $\fd(\Hom_R(T,I)_S)\leq \id(_ST)$ for every injective $R$-module $I_R$.
\end{itemize}
\end{lemma}

\begin{proof} (i) We can assume that $\id(T_R)$ is finite. Let $F_S$ be a flat S-module. By \cite[Theorem 3.2.15]{EJ},
we have $\Ext^i_R(A,F\otimes_ST)=0$ for any finitely presented $R$-module $A_R$ and all $i>\id(T_R)$. This specifically
applies to $A=R/\frak a$ for any right ideal $\frak a$ of $R$, which, by \cite[Theorem 3.1.9]{EJ}, implies that
$\id((F\otimes_ST)_R)\leq \id(T_R)$.

(ii) We can assume that $\id(_ST)$ is finite. Let $I_R$ be an injective $R$-module. From \cite[Theorem 3.2.13 and
Remarks 3.2.25 and 3.2.27]{EJ}, we know that $\Tor_i^S(\Hom_R(T,I),A)=0$ for any finitely presented $S$-module
$_SA$ and all $i>\id(_ST)$. Taking $A=S/\frak b$ for any finitely generated left ideal $\frak b$ of $S$, by
\cite[Theorem 3.2.10]{EJ}, we conclude that $\fd(\Hom_R(T,I)_S)\leq \id(_ST)$.
\end{proof}

\begin{definition}\label{2.8} A bimodule $_ST_R$ is called {\it tensorly faithful} if it satisfies the
following conditions for all modules $_RM$ and $N_S$.
\begin{itemize}
\item[(i)] If $T\otimes_RM=0$, then $M=0$.
\item[(ii)] If $N\otimes_ST=0$, then $N=0$.
\end{itemize}
\end{definition}

The concept of tensorly faithful modules was originally introduced in \cite{HW} under the name
\enquote{faithfully modules}. To avoid ambiguity with faithful modules, we renamed this notion.

Note that, if $R$ is a commutative Noetherian ring and $S=R$, then every Wakamatsu tilting $R$-module is
tensorly faithful; see \cite[Proposition 3.1]{HW}. We are unaware if this result carries over to the
non-commutative setting. Also, in \cite{HW} many examples of tensorly faithful Wakamatsu tilting modules
were provided over a wide class of non-commutative rings.

\begin{lemma}\label{2.9} Let $_ST_R$ be a Wakamatsu tilting module. If $T$ is tensorly faithful, then
the classes $\mathscr{A}_T\left(S\right)$ and $\mathscr{B}_T\left(R\right)$ have the property that if
two of three modules in a short exact sequence are in the class, then so is the third.
\end{lemma}

\begin{proof} See \cite[Corollary 6.3]{HW}.
\end{proof}

To conclude this section, we provide a brief review of the notion of cotilting modules as defined by Miyashita.

\begin{definition}\label{2.10} (See \cite[page 142]{MI}.)  Let $C_R$ be a finitely generated $R$-module
and $S=\End(C_R)$. Assume that the ring $R$ is right Noetherian and the ring $S$ is left Noetherian. We say
$C_R$ is {\it cotilting} if it is a Wakamatsu tilting $R$-module and both $\id(C_R)$ and $\id(_SC)$ are finite.
\end{definition}

The above notion of cotilting modules is also examined in \cite{M2, M1, HT}.

\begin{example}\label{2.11}
\begin{itemize}
\item[(i)] Let $R$ be a connected finite-dimensional hereditary algebra over an algebraically closed field $\Bbbk$ of
infinite representation type. Let $M$ be a finitely generated $R$-module, we define $\Tr~(M):=\Coker(\Hom_R(f,R)),$
where $P_1\st{f}\rt P_0\rt M\rt0$ is a minimal projective presentation of $M$, and $D(-):=\Hom_{\Bbbk}(-,\Bbbk)$. Then
$(\Tr~ D)^n(R)$ is a cotilting $R$-module for all $n\geq 0$; see \cite[page 592]{M2}.
\item[(ii)] Let $\Gamma$ be a commutative Cohen-Macaulay local ring with a dualizing module $\omega$, and $R$ be a
$\Gamma$-order (i.e. $R$ is a $\Gamma$-algebra and $R$ is a finitely generated maximal Cohen-Macaulay as a $\Gamma$-module).
Then, by \cite[Proposition 2.12]{M2}, $\Hom_\Gamma(R,\omega)$ is a cotilting $R$-module.
\item[(iii)] Let $R$ be an Artin algebra and $C_R$ a cotilting $R$-module with $\id(C_R)=r$. Assume that $R\mathcal{Q}$
is the path algebra of the quiver $\mathcal{Q}:1\rightarrow 2 \rightarrow 3$. It is easy to see that $R\mathcal{Q}$ is
an Artin algebra. By \cite[Lemma 3.7]{Zh}, $$\textbf{C}:=(0\rightarrow0 \rightarrow C)\oplus (0\rightarrow C \rightarrow C)
\oplus (C\rightarrow C\rightarrow C)$$ is a cotilting $R\mathcal{Q}$-module with $\id(\textbf{C}_{R\mathcal{Q}})=r+1$.
\end{itemize}
\end{example}

\section{Auslander classes}

We begin this section by reviewing the definitions of Gorenstein projective and Gorenstein flat modules.

\begin{definition}\label{3.1}
\begin{itemize}
\item[(i)] An $S$-module $N_S$ is said to be {\it Gorenstein projective} if there exists an exact complex
$${\bf P}:\cdots \rt P_1\rt P_0 \rt P^0 \rt P^1 \rt \cdots$$ of projective right $S$-modules such that
$N\cong \Ker(P^0\rt P^1)$, and $\Hom_S({\bf P},P)$ is exact for every projective $S$-module $P_S$. The exact
complex ${\bf P}$ is referred to as a {\it complete projective resolution} of $N$.
\item[(ii)] An $S$-module $N_S$ is said to be {\it Gorenstein flat} if there exists an exact complex
$${\bf F}:\cdots \rt F_1 \rt F_0 \rt F^0\rt F^1 \rt \cdots$$ of flat right $S$-modules such that $N\cong
\Ker(F^0\rt F^1)$, and ${\bf F}\otimes_SI$ is exact for every injective $S$-module $_SI$. The exact complex
${\bf F}$ is referred to as a {\it complete flat resolution} of $N$.
\end{itemize}
\end{definition}

\begin{definition}\label{3.2} The Gorenstein projective (respectively, Gorenstein flat) dimension of a nonzero
$S$-module $N_S$, denoted by $\Gpd(N_S)$ (respectively, $\Gfd(N_S)$), is the least non-negative integer $n$ such
that there exists an exact complex $$0\rt A_n\rt \cdots \rt A_1 \rt A_0 \rt N\rt 0,$$ where each $A_i$ is a
Gorenstein projective (respectively, Gorenstein flat) right $S$-module.
\end{definition}

\begin{lemma}\label{3.3} Let $S$ be a left coherent ring and $N_S$ an $S$-module. Let $\CL$ be the class of all
right $S$-modules with finite flat dimension. If $\Ext^i_S(N,L)=0$ for every $L\in \CL$ and all $i>0$, then there
exists a $\CL$-preenvelope $N\rt P$, where $P_S$ is a projective $S$-module.
\end{lemma}

\begin{proof} As $S$ is a left coherent, there exists a flat-preenvelope $\alpha: N\rt F$ by \cite[Proposition 6.5.1]{EJ}.
Consider a short exact sequence $$0\rt W \rt P \st{\beta}\rt F\rt 0,$$ where $P_S$ is a projective $S$-module. From this
sequence, it easily follows that $W$ is flat. Then, by the assumption $\Ext^1_S(N,W)=0$. Hence, $$\Hom_S(N,\beta):
\Hom_S(N,P)\rt \Hom_S(N,F)$$ is surjective. So, $\alpha:N\rt F$ is lifted to a map $f:N\rt P$, i.e. $\alpha=\beta f$. This
implies that $$\Hom_S(\alpha,F')=\Hom_S(f,F')\Hom_S(\beta,F')$$ for every flat right $S$-module $F'$. Since $\Hom_S(\alpha,F')$
is surjective, then so does $\Hom_S(f,F')$. Consequently, $f:N\rt P$ is a flat-preenvelope.

Next, we show that $f:N\rt P$ is an
$\CL$-preenvelope. Indeed, we have to show that for every $Y\in \CL$, the map $$\Hom_S(f,Y):\Hom_S(P,Y)\rt \Hom_S(N,Y)$$ is
surjective. Let $Y\in \CL$. We can consider a short exact sequence $$0\rt K\rt Q\rt Y\rt 0,$$ in which $Q_S$ is projective and
$K\in \CL$. Applying the functor $\Hom_S(N,-)$ to it induces the following exact sequence $$0\rt\Hom_S(N,K)\rt \Hom_S(N,Q)\rt
\Hom_S(N,Y)\rt 0.$$ Thus, we have the following commutative diagram
$$
\begin{CD}
\ \Hom_S(P,Q) @>>> \Hom_S(P,Y) \\
 @V VV @V VV \\ \Hom_S(N,Q) @>>> \Hom_S(N,Y).
\end{CD}
$$
As the left vertical map and the bottom map are surjective, it follows that the right vertical map is also surjective.
\end{proof}

\begin{lemma}\label{3.4} Let $_ST_R$ be a Wakamatsu tilting module, where $R$ is a right Noetherian ring. Then for every
flat $S$-module $F_S$, we have $\pd(F_S)\leq \id(T_R)$.
\end{lemma}

\begin{proof} We can assume $\id(T_R)$ is finite. Let $n=\id(T_R)$. Let $${\bf{Q_\bullet}}:\cdots \rt P_{i-1}\rt P_{i-2} \rt
\cdots \rt P_1\rt P_0 \rt F \rt 0$$ be a projective
resolution of $F_S$, and denote the complex $$\cdots \rt P_{i-1} \rt P_{i-2} \rt\cdots \rt P_1 \rt P_0 \rt 0$$ by
${\bf{P_\bullet}}$. Putting, for simplicity, $P_{-1}=F,$  $K_0=F$ and $K_i=\Ker(P_{i-1}\rt P_{i-2})$ for all $i\geq 1$.
Splitting ${\bf{Q_\bullet}}$ into short exact sequences yields that each $K_i$ is flat, and so $K_i\in \mathscr{A}_T
\left(S\right)$ by Lemma \ref{2.4}(1).  To obtain the desired result, it suffices to show that the short
exact sequence $$0\rt K_{n+1}\rt P_n \rt K_n \rt 0$$ splits. Indeed, we need to show that $\Ext^1_S( K_n, K_{n+1})=0$.
We have the following isomorphisms:
\[\begin{array}{lllll}
\Ext^1_{S}(K_n,K_{n+1})& \cong \Ext^{n+1}_{S}(F,K_{n+1})\\
& \cong \Ext^{n+1}_{S}(F,\Hom_{R}(T,K_{n+1}\otimes_S T)) \\
& = \HH^{n+1}( \Hom_S({\bf{P_\bullet}},\Hom_{R}(T,K_{n+1}\otimes_S T)))\\
& \cong \HH^{n+1}( \Hom_R({\bf{P_\bullet}}\otimes_{S}T,K_{n+1}\otimes_S T))\\
& \cong \Ext^{n+1}_R(F\otimes_{S}T,K_{n+1}\otimes_S T).
\end{array}\]
The first isomorphism follows by dimension shift and the second isomorphism holds, because $K_{n+1}\in \mathscr{A}_T
\left(S\right)$. The third isomorphism uses of tensor-hom adjunction. The fourth isomorphism follows from the fact
the complex ${\bf{Q_\bullet}}\otimes_ST$ is exact, and also $$\Ext^{i\geq1}_R(P\otimes_{S}T,K_{n+1}\otimes_S T)\cong \Ext^{i\geq1}_S(P,K_{n+1})=0$$ for
every projective $S$-module $P_S$ by Lemma \ref{2.4}(4). Since, by Lemma \ref{2.6}(i), $$\id((K_{n+1}\otimes_S T)_R)\leq
\id(T_R)=n,$$ we conclude that $$\Ext^{n+1}_R(F\otimes_{S}T,K_{n+1}\otimes_S T)=0.$$ This completes the proof.
\end{proof}

Part (i) of the next result provides an answer to \cite[Question 2]{GZ}.

\begin{lemma}\label{3.8} Let $_ST_R$ be a Wakamatsu tilting module and $N_S$ a Gorenstein flat $S$-module. If $_ST$ has
finite projective dimension or finite injective dimension, then $N_S\in \mathscr{A}_T\left(S\right)$.
\end{lemma}

\begin{proof} As $N_S$ is a Gorenstein flat $S$-module, there is an exact complex of flat modules $${\bf{F_\bullet}}:\cdots
\rt F_1\rt F_0\rt F^0\rt F^1\rt \cdots$$ such that $N\cong \Ker(F^0\rt F^1)$ and it remains exact after applying $-\otimes_SI$
for every injective $S$-module $_SI$. It readily follows that for every $S$-module $_SL$ such that either $\pd(_SL)<\infty$
or $\id(_SL)<\infty$, the complex ${\bf{F_\bullet}}\otimes_RL$ is exact, and so $\Tor_i^S(N,L)=0$ for all $i\geqslant 1$.
In particular, $\Tor_i^S(N,T)=0$ for all $i\geqslant 1$.

Since the complex ${\bf{F_\bullet}}\otimes_RT$ is exact, we have the following exact sequence
\begin{equation}
0\rt N\otimes_ST \rt F^0\otimes_ST \rt F^1\otimes_ST. \label{2}
\end{equation}
Applying the functor $\Hom_R(T,-)$ to \eqref{2}, gives rise to the following commutative diagram with exact rows:

\vspace{0.3cm}
\begin{tikzcd}
0\arrow{r} &N\arrow{r}\arrow{d}{\theta_{N}^T} &F^0\arrow{r}\arrow{d}{\theta_{F^0}^T} &F^1\arrow{d}{\theta_{F^1}^T} \\	
0\arrow{r} &\Hom_R(T,N\otimes_ST)\arrow{r}	&\Hom_R(T,F^0\otimes_ST)\arrow{r} &\Hom_R(T,F^1\otimes_ST)
\end{tikzcd}
\vspace{0.4cm}\\
But each $\theta_{F^i}^T$ is an isomorphism, because $F^i\in \mathscr{A}_T\left(S\right)$ for $i=0,1$.  Consequently,
$\theta_{N}^T$ is also an isomorphism.

For showing that $N\in  \mathscr{A}_T\left(S\right)$, it remains to prove that $\Ext^i_R(T,N\otimes_ST)=0$ for all $i\geq 1$.
To this end, by induction on $i$, we show that $\Ext^i_R(T,G\otimes_ST)=0$ for every Gorenstein flat $S$-module $G_S$
and all $i\geq 1$. Let $G_S$ be a Gorenstein flat $S$-module. We can consider a short exact sequence $$0\rt G
\rt F \rt K \rt 0,$$ in which $F_S$ is a flat $S$-module and $K_S$ is a Gorenstein flat $S$-module. From the
above, we get that
\begin{equation}
0\rt G\otimes_ST \rt F\otimes_ST \rt K\otimes_ST\rt 0.  \label{3}
\end{equation}
is exact and that for every Gorenstein flat $S$-module $L_S$, the natural map $$\theta_L^T: L\rt \Hom_R(T,L\otimes_ST)$$
is an isomorphism. Thus applying the functor $\Hom_R(T,-)$ to \eqref{3}, yields the following long exact sequence
$$0\rt G\rt F\rt K \rt \Ext^1_R(T,G\otimes_ST)\rt \Ext^1_R(T,F\otimes_ST)\rt \cdots$$
$$\cdots \rt \Ext^{i-1}_R(T,F\otimes_ST)\rt \Ext^{i-1}_R(T,K\otimes_ST)\rt \Ext^{i}_R(T,G\otimes_ST)\rt
\Ext^{i}_R(T,F\otimes_ST)\rt \cdots.$$
But $\Ext^{i\geq 1}_R(T,F\otimes_ST)=0$, since $F\in \mathscr{A}_T\left(S\right)$. Thus $\Ext^1_R(T,G\otimes_ST)=0$ and
$$\Ext^{i}_R(T,G\otimes_ST)\cong \Ext^{i-1}_R(T,K\otimes_ST)$$ for all $i\geq 2$. So, the claim holds for $i=1$.

Assume that $i\geq 2$ and $\Ext^{i-1}_R(T,K\otimes_ST)=0$ for all Gorenstein flat $S$-modules $K_S$. Then from the
above isomorphism and the induction hypothesis, it follows that $\Ext^{i}_R(T,G\otimes_ST)=0$.
\end{proof}

Lemma \ref{2.9} yields:

\begin{corollary}\label{3.9} Let $_ST_R$ be a tensorly faithful Wakamatsu tilting module and $N_S$ an $S$-module
with finite Gorenstein flat dimension. If $_ST$ has finite projective dimension or finite injective dimension, then
$N_S\in \mathscr{A}_T\left(S\right)$.
\end{corollary}

As the last preparation to prove the main result of this section, we recall the definition of $C$-injective modules.
Given a Wakamatsu tilting module $C_R$ with $S=\End(C_R)$, an $S$-module $M_S$ is said to be $C$-{\it injective} if $M$
has the form $\Hom_R(C,I)$ for some injective $R$-module $I_R$.

\begin{theorem}\label{3.5} Let $C_R$ be a tensorly faithful cotilting module, and let $S=\End(C_R)$. Then, for
an $S$-module $N_S$, the following are equivalent:
\begin{itemize}
\item[(i)] $N_S\in \mathscr{A}_C\left(S\right)$.
\item[(ii)] $\Gpd(N_S)$ is finite.
\item[(iii)] $\Gfd(N_S)$ is finite.
\end{itemize}
\end{theorem}

\begin{proof} (i)$\Rightarrow $(ii)  Let $$\cdots \rt P_{i} \st{f_i}\rt  P_{i-1}\rt \cdots \rt P_1 \st{f_1}\rt
P_0\st{f_0}\rt N\rt 0$$ be a projective resolution of $N$ and let $n=\id(C_R)$. In order to obtain the
desired result, it suffices to show that $D:=\Ker (f_{n})$ is a Gorenstein projective $S$-module. To do this,
by \cite[Proposition 2.3]{Ho}, we have to show that $\Ext^{i\geq1}_{S}(D,P)=0$ for all projective $S$-module $P_S$
and there exists an exact complex $${\bf{P^\bullet}}: 0 \rt D \rt P^0 \rt P^1 \rt P^2 \rt \cdots,$$ in which each
$P^i$ is a projective right $S$-module and it remains exact after applying $\Hom_S(-,Q)$ for every projective
$S$-module $Q_S$.

For the first assertion, it suffices to show that $ \Ext^i_{S}(D,F)=0$ for every flat $S$-modules $F_S$ and all
$i\geqslant 1$. Let $F_S$ be a flat $S$-module. For each $i\geqslant 1$, we have $$\Ext^i_{S}(D,F)\cong
\Ext^{i+n+1}_{S}(N,F)\cong \Ext^{i+n+1}_{R}(N\otimes_SC,F\otimes_SC).$$ The first isomorphism follows from
dimension shift and the second isomorphism holds by Lemma \ref{2.4}(4). Since $R$ is right Noetherian, we may apply
\ref{2.6}(i) to obtain $$\id((F\otimes_SC)_R)\leq \id(C_R)=n.$$ Hence,
\begin{equation}
\Ext^i_{S}(D,F)=0 \label{0}
\end{equation}
for all $i\geqslant 1$.

Next, we prove the second assertion. For any $S$-module $L_S$ with finite flat dimension, we may choose a short
exact sequence $0\rt K\rt F\rt L\rt 0$, where $F_S$ is a free $S$-module and $\fd(K_S)=\fd(L_S)-1$. This induces
a long exact sequence of Ext modules: $$\cdots \rt \Ext^i_{S}(D,K)\rt \Ext^i_{S}(D,F)\rt \Ext^i_{S}(D,L)\rt
\Ext^{i+1}_{S}(D,K)\rt \cdots.$$ Using \eqref{0} and induction on flat dimension, we conclude that
$\Ext^{i\geq1}_{S}(D,L)=0$ for all $L\in \CL=\lbrace L\mid \fd(L_S)<\infty\rbrace$. So, by Lemma \ref{3.3}, there
exists a projective right $S$-module $P^0$ and $S$-homomorphism $\mu: D\rt P^0$, which is an $\CL$-preenvelope. As
$N$ and $P_i$s are in $\mathscr{A}_C\left(S\right)$, by Lemma \ref{2.4}(1) so is $D$. Since $D$ belongs to
$\mathscr{A}_C\left(S\right)$, by \cite[Theorem 2]{HW}, there exists an exact sequence of right $S$-modules
$$\cdots \rt P_1\rt P_0\rt U^0\rt U^1\rt \cdots,$$ where each $P_i$ is projective, each $U^i$ is $C$-injective
and $D\cong \Coker( P_1\rt P_0)$. Thus, there exists an injective $R$-module $I_R$ such that $D$ embeds in
$\Hom_R(C,I)$. On the other hand, since $S$ is left Noetherian, Lemma \ref{2.6}(ii) yields $$\fd(\Hom_R(C,I)_S)\leq
\id(_SC)<\infty.$$ (Note that since $C_R$ is cotilting, it follows by definition that $_SC$ has finite injective
dimension.) Hence, $\Hom_R(C,I)\in \CL$. As $\Hom_R(C,I)\in \CL$, it follows that $\mu$ is monic. Now, we consider
the short exact sequence
\begin{equation}
0\rt D\st{\mu}\rt P^0\rt \Coker \ \mu \rt 0. \label{1}
\end{equation}
Since $P^0$ and $D$ are in $\mathscr{A}_C\left(S\right)$, Lemma \ref{2.9} implies that $\Coker \ \mu\in
\mathscr{A}_C\left(S\right)$. Let us take $L\in\CL$, applying the functor $\Hom_S(-,L)$ to \eqref{1}, gives rise
to the following exact sequence: $$0\rt \Hom_S(\Coker \ \mu,L)\rt \Hom_S(P^0,L)\st{f}\rt\Hom_S(D,L)\rt
\Ext^1_S(\Coker \ \mu,L)\rt
0,$$ and the isomorphisms $$\Ext^{i+1}_S(\Coker \ \mu,L)\cong \Ext^i_S(D,L)$$ for all $i\geqslant 1$. As the class
$\mathcal{L}$ is preenveloping, the map $f$ must be surjective, which consequently yields $\Ext^1_S(\Coker \ \mu,L)=0$.
Thus, $\Ext^i_S(\Coker \ \mu,L)=0$ for all $i \geqslant 1$. So, by the above argument, there exist a projective $S$-module
$P^1$ and an $S$-monomorphism $\Coker \ \mu\rt P^1$ which is an $\CL$-preenvelope. We proceed in this manner to construct
an exact complex $$X=0\rt D \overset{d^{-1}}\longrightarrow P^0\overset{d^0}\longrightarrow P^1 \overset{d^1}\longrightarrow
P^2\rt \cdots ,$$ where, for each $i\geq 0$, $P^i$ is a projective right $S$-module and $\Ext^1_S(\Coker \
d^{i-1},L)=0$. This implies that $X$ remains exact after applying $\Hom_S(-,Q)$ for every projective $S$-module $Q_S$.

(ii) $\Rightarrow$ (iii) By Lemma \ref{3.4}, every flat $S$-module $F_S$ has finite projective dimension. This fact
and the proof of \cite[Proposition 3.4]{Ho} yield that every Gorenstein projective right $S$-module is Gorenstein
flat. So, $\Gfd(N_S)$ is finite.

(iii) $\Rightarrow$ (i) holds by Corollary \ref{3.9}.
\end{proof}

\section{Bass classes}

We start this section by reviewing the definition of Gorenstein injective modules.

\begin{definition}\label{4.1}  An $R$-module $M_R$ is said to be {\it Gorenstein injective} if there exists an exact
complex $${\bf E}:\cdots \rt E_1\rt E_0 \rt E^0\rt E^1 \rt \cdots$$ of injective right $R$-modules such that $M\cong
\Ker(E^0\rt E^1)$, and $\Hom_R(I,{\bf E})$ is exact for every injective $R$-module $I_R$. The exact complex ${\bf E}$
is called a {\it complete injective resolution} of $M$.
\end{definition}

\begin{definition}\label{4.2} The {\it Gorenstein injective dimension} of a nonzero $R$-module $M_R$, denoted by
$\Gid(M_R)$, is the least non-negative integer $n$ such that there exists an exact complex $$0\rt M \rt B^0 \rt
B^1\rt\cdots \rt B^n \rt 0,$$ where each $B^i$ is a Gorenstein injective right $R$-module.
\end{definition}

\begin{lemma}\label{4.3} Let $R$ be a right Noetherian ring and $M_R$ an $R$-module. Let $\CE$ denote the class
of all right $R$-modules with finite injective dimension. If $\Ext^i_R(L,M)=0$ for all $L\in \CE$ and all $i>0$,
then there exists an $\CE$-precover $E\rt M$, where $E_R$ is an injective $R$-module.
\end{lemma}

\begin{proof} Since $R$ is a right Noetherian,  by \cite[Theorem 5.4.1]{EJ}, there exists an injective-precover
$\alpha:E\rt M$. We have to show that the map $$\Hom_R(X,\alpha):\Hom_R(X,E)\rt \Hom_R(X,M)$$ is surjective for
every $X\in \CE$. Let $X\in \CE$. We can consider a short exact sequence $$0\rt X \rt E'\rt W\rt 0,$$ where $E'$
is an injective right $R$-module and $W\in \CE$. Applying the functor $\Hom_R(-,M)$ to this short exact sequence,
yields the following exact sequence $$0\rt \Hom_R(W,M)\rt \Hom_R(E',M)\rt \Hom_R(X,M)\rt 0.$$ Thus, we have the
following commutative diagram
$$\begin{CD}
 \ \Hom_R(E',E) @>>> \Hom_R(X,E) \\
@V VV @V VV \\ \Hom_R(E',M) @>>> \Hom_R(X,M).
\end{CD}$$
As the left vertical map and the bottom map are surjective, it follows that the right vertical map is also surjective.
\end{proof}

Recently, Gao and Zhao \cite[Theorem 4.11]{GZ} proved that if $_ST_R$ is tensorly faithful Wakamatsu tilting module
with $\pd(T_R)<\infty$, then every Gorenstein injective right $R$-module is in $ \mathscr{B}_T\left(R\right)$. Next, we show that
the assumption \enquote{tensorly faithfulness} is unnecessary.

\begin{lemma}\label{4.5} Let $_ST_R$ be a Wakamatsu tilting module and $M_R$ a Gorenstein injective $R$-module.
If $T_R$ has finite projective dimension or finite injective dimension, then $M_R\in \mathscr{B}_T\left(R\right)$.
\end{lemma}

\begin{proof} Let $${\bf{E_\bullet}}:\cdots \rt E_1 \rt E_0\rt E^0 \rt E^1 \rt \cdots$$ be a complete injective resolution
of $M$. So, $M\cong \Ker(E^0\rt E^1)$. It easily follows that for each $R$-module $L_R$ such that either $\pd(L_R)<\infty$
or $\id(L_R)<\infty$, the complex $\Hom_R(L,{\bf{E_\bullet}})$ is exact, and so $\Ext^i_R(L,M)=0$ for all $i\geqslant 1$.
In particular, $\Ext^i_R(T,M)=0$ for all $i\geqslant 1$.

Since the complex $\Hom_R(T,{\bf{E_\bullet}})$ is exact, the sequence
\begin{equation}
\Hom_R(T,E_1)\rt \Hom_R(T,E_0)\rt \Hom_R(T,M)\rt 0 \label{4}
\end{equation}
is exact. Applying the functor $-\otimes_ST$ to \eqref{4} gives rise to the following commutative diagram with exact rows:

\vspace{0.3cm}
\begin{tikzcd}
\Hom_R(T,E_1)\otimes_ST\arrow{r}\arrow{d}{\nu_{E_1}^T}	&\Hom_R(T,E_0)\otimes_ST\arrow{r}\arrow{d}{\nu_{E_0}^T}	
&\Hom_R(T,M_1)\otimes_ST\arrow{r}\arrow{d}{\nu_{M}^T} &0\\
E_1\arrow{r} &E_0\arrow{r} &M\arrow{r} &0
\end{tikzcd}
\vspace{0.3cm}\\
Each $\nu_{E_i}^T$ is an isomorphism, because $E_i\in \mathscr{B}_T\left(R\right)$ for $i=0,1$. Consequently, $\nu_{M}^T$
is also an isomorphism.

To complete the proof, by induction on $i$,  we show that $\Tor^S_{i}(\Hom_R(T,G),T)=0$ for every Gorenstein injective
$R$-module $G_R$
and all $i\geq 1$. Let $G_R$ be a Gorenstein injective $R$-module. We can consider a short exact sequence $$0\rt K \rt
E \rt G\rt 0,$$ where $E_R$ is an injective $R$-module and $K_R$ is a Gorenstein injective $R$-module. From the above,
we get that
\begin{equation}
0\rt \Hom_R(T,K)\rt \Hom_R(T,E) \rt\Hom_R(T,G)\rt 0 \label{5}
\end{equation}
is exact and that for every Gorenstein injective $R$-module $L_R$, the natural map $$\nu_L^T:\Hom_R(T,L)\otimes_ST\rt L$$
is an isomorphism. Hence applying the functor $-\otimes_ST$ to \eqref{5}, implies the following long exact sequence
$$\cdots \rt \Tor^S_{1}(\Hom_R(T,E),T) \rt \Tor^S_{1}(\Hom_R(T,G),T) \rt K \rt E\rt G\rt 0$$
$$\rt \Tor^S_{i}(\Hom_R(T,G),T) \rt \Tor^S_{i-1}(\Hom_R(T,K),T)\rt \Tor^S_{i-1}(\Hom_R(T,E),T)$$
$$\cdots\rt \Tor^S_{i}(\Hom_R(T,E),T)$$
But $\Tor^S_{i\geq 1}(\Hom_R(T,E),T)=0$, since $E\in \mathscr{B}_T\left(R\right)$. Thus, $\Tor^S_{1}(\Hom_R(T,G),T)=0$ and
$$\Tor^S_{i}(\Hom_R(T,G),T)\cong\Tor^S_{i-1}(\Hom_R(T,K),T)$$ for all $i\geq 2$. In particular, the claim hold for
$i=1$.

Assume that $i\geq 2$ and $\Tor^S_{i-1}(\Hom_R(T,K),T)=0$ for all Gorenstein injective $R$-modules $K_R$. Then from
the above isomorphism and the induction hypothesis, it follows that $\Tor^S_{i}(\Hom_R(T,G),T)=0$.
\end{proof}

As an immediate consequence of Proposition \ref{4.5} and Lemma \ref{2.9}, we have the following result:

\begin{corollary}\label{4.6}  Let $_ST_R$ be a tensorly faithful Wakamatsu tilting module and $M_R$ an $R$-module with
finite Gorenstein injective dimension. If $T_R$ has finite projective dimension or finite injective dimension, then
$M_R\in \mathscr{B}_T\left(R\right)$.
\end{corollary}

We are now ready to prove the main result of this section. Since the proof involves the notion of $C$-projective modules,
we need to recall their definition. Let $C_R$ be a Wakamatsu tilting module with $S = \End(C_R)$. An  $R$-module $M_R$ is
called $C$-{\it projective} if it is isomorphic to $P\otimes_SC$ for some projective right $S$-module $P_S$.

\begin{theorem}\label{4.4} Let $C_R$ be a tensorly faithful cotilting module, and let $S=\End(C_R)$. Then, for an
$R$-module $M_R$, the following are equivalent:
\begin{itemize}
\item[(i)] $M_R\in \mathscr{B}_C\left(R\right)$.
\item[(ii)] $\Gid(M_R)$ is finite.
\end{itemize}
\end{theorem}

\begin{proof} (i)$\Rightarrow$(ii) Let $$0\longrightarrow M \st{g^{-1}} \longrightarrow E^0 \st{g^0}\longrightarrow E^1
\longrightarrow \cdots \longrightarrow E^{i-1} \st{g^{i-1}}\longrightarrow E^{i}\longrightarrow \cdots$$ be an injective
resolution of $M$, $n=\id(C_R)$ and $d=\id(_SC)$. To conclude the desired result, it suffices to show that
$D:=\Coker (g^{(n+d)-1})$ is a Gorenstein injective right $R$-module.

For this purpose, it suffices to show that $\Ext^{i\geq 1}_{R}(I,D)=0$ for every injective $R$-module $I_R$ and that there
exists an exact complex $${\bf{E^\bullet}}:\cdots \rt E_n \rt E_{n-1}\rt \cdots \rt E_1 \rt E_0 \rt D \rt 0,$$ in which each
$E_i$ is an injective right $R$-module and it remains exact after applying $\Hom_R(I,-)$ for every injective $R$-module $I_R$.

For the first assertion, let $I_R$ be an injective $R$-module. For each $i\geqslant 1$, we have $$\Ext^i_{R}(I,D)\cong \Ext^{i+(d+n)+1}_{R}(I,M)\cong\Ext^{i+(d+n)+1}_{S}(\Hom_R(C,I),\Hom_R(C,M)).$$ The first isomorphism follows from dimension
shift and the second isomorphism holds by Lemma \ref{2.4}(5). Since $S$ is left Noetherian, Lemma \ref{2.6}(ii) yields
$$\fd(\Hom_R(C,I)_S)\leq \id(_SC)=d.$$ So, by applying Lemma \ref{3.4}, we conclude that $$\pd(\Hom_R(C,I)_S)\leq n+d.$$
Consequently, $\Ext^i_R(I,D)=0$ for all $i\geqslant 1$. From this, we can easily deduce that $$\Ext^{i\geq 1}_{R}(L,D)=0$$
for every $L\in \CE$, where $\CE$ is the class of all right $R$-modules with finite injective dimension.

Next, we establish the second assertion. By Lemma \ref{4.3}, there exist an injective right $R$-module $E_0$ and an
$R$-homomorphism $\rho: E_0\rt D$, which is an $\CE$-precover. As $M$ and $E^i$s are in $ \mathscr{B}_C\left(R\right)$,
by Lemma \ref{2.4}(2) so is $D$. Since $D$ belongs to $\mathscr{B}_C\left(S\right)$, by \cite[Theorem 6.1]{HW}, there
exists an exact sequence of right $R$-modules $$\cdots \rt W_1\rt W_0\rt I^0\rt I^1\rt \cdots,$$ where each $I^i$ is
injective, each $W_i$ is $C$-projective and $D\cong \Ker(I^0\rt I^1)$. Thus, there exists a projective $S$-module $P_S$
and an epimorphism $P\otimes_SC\rt D$. On the other hand, since $R$ is right Noetherian, Lemma \ref{2.6}(i) implies
$$\id_R((P\otimes_SC)_R)\leq \id(C_R)< \infty.$$ (Recall that since $C_R$ is cotilting, it has finite injective dimension
by definition.) Hence, $P\otimes_SC\in \CE$. As $P\otimes_SC\in \CE$, it follows that the map $\rho$ is epic. Now, we consider
the short exact sequence
\begin{equation}
0\rt K\rt E_0\st{\rho} \rt D\rt 0.  \label{w}
\end{equation}
Since $E_0$ and $D$ are in $\mathscr{B}_C\left(R\right)$, Lemma \ref{2.9} implies that $K\in \mathscr{B}_C\left(R\right)$.
Let us take $L\in \CE$, applying the functor $\Hom_R(L,-)$ to \eqref{w} gives rise to the following exact sequence $$0\rt
\Hom_R(L,K)\rt \Hom_R(L,E_0)\st{g}\rt\Hom_R(L,D)\rt \Ext^1_R(L,K)\rt 0,$$ and the isomorphisms
$$\Ext^{i+1}_R(L,K)\cong \Ext^i_R(L,D)$$ for all $i\geqslant 1$. As the class $\CE$ is precovering, the map $g$ must be
surjective, which consequently yields $\Ext^1_R(L,K)=0$.
Thus, $\Ext^i_R(L,K)=0$ for all $i\geqslant 1$. Hence, by the above argument, there exists an injective right $R$-module
$E_1$ and an epimorphism $E_1\rt K$ which is $\CE$-precover. We proceed in this manner to construct an exact complex
$$X=\cdots \rt E_2 \overset{d_2}\longrightarrow E_1 \overset{d_1}\longrightarrow E_0 \overset{d_0}\longrightarrow D \rt 
0,$$ where for each $i\geq 0$,
$E_i$ is an injective right $R$-module and $\Ext^1_R(L,\Ker \ d^i)=0$. This implies that $X$ remains exact after applying
$\Hom_R(I,-)$ for every injective $R$-module $I_R$.

(ii) $\Rightarrow$ (i) follows by Corollary \ref{4.6}.
\end{proof}

\begin{remark}\label{4.7} Considering Theorems \ref{3.5} and \ref{4.4}, one may wonder if we can characterize the classes $\mathscr{A}_T\left(S\right)$ and $\mathscr{B}_T\left(R\right)$ in the case $_ST_R$ is a Wakamatsu tilting module with
$\pd(T_R)$ and $\pd(_ST)$ finite. It is straightforward to observe  that if $_ST_R$ is a Wakamatsu tilting module
with $\pd(T_R)<\infty$ and $\pd(_ST)<\infty$, then $_ST_R$ is a tilting module. Furthermore, by \cite[Lemma 1.21]{MI}, for
any tilting module $_ST_R$, it is known that $\mathscr{A}_T\left(S\right)=\{N\in \text{Mod-}S\mid \Tor^S_{i\geq 1}(N,T)=0\}$
and $\mathscr{B}_T\left(R\right)=\{M\in \text{Mod-}R\mid \Ext_R^{i\geq 1}(T,M)=0\}$.
\end{remark}

\begin{example}\label{4.8} Let $\Gamma$ be a commutative Gorenstein ring and $\mathcal{Q}$ a finite acyclic quiver.
By \cite[Corollary 2.14]{BMS} the path algebra $\Gamma\mathcal{Q}$ is Iwanaga–Gorenstein. Take $T=R=\Gamma\mathcal{Q}$.
Then $S=\End_R(T)=R$ and $T$ is a tilting module. Specifically, $T$ is a tensorly faithful cotilting module. Therefore,
by Remark \ref{4.7}, we have $\mathscr{A}_T\left(S\right)=\mathscr{B}_T\left(R\right)=\text{Mod-}R$.
\end{example}

We end the paper by the following remark on the assumption \enquote{tensorly faithfulness} in Theorems \ref{3.5} and
\ref{4.4}.

\begin{remark}\label{4.9} Since every dualizing module over a Noetherian commutative ring is tensorly faithful, the
assumption of \enquote{tensorly faithfulness} is not required in the commutative ring analogues of Theorems \ref{3.5}
and \ref{4.4}. This naturally leads to the following four questions:
\begin{itemize}
\item[(i)] Under the assumptions of Theorems \ref{3.5} and \ref{4.4}, is any cotilting $R$-module $C_R$ necessarily
tensorly faithful? This is not the case. To this end, let $\Bbbk$ be a field and $R=\Bbbk\mathcal{Q}$ be the
finite-dimensional $\Bbbk$-algebra that corresponds to the quiver $$\mathcal{Q}:1\rightarrow 2 \rightarrow 3.$$ It is
easy to see that the right $R$-module $C=P(1)\oplus P(3)\oplus S(3)$ is both tilting and cotilting, and it is not
projective. As $C_R$ is a non-projective tilting $R$-module, by \cite[Theorem 2.11]{DMT}, $C_R$ couldn't be tensorly
faithful.
\item[(ii)] Are there sufficient examples of tensorly faithful cotilting modules over non-commutative, Artinian rings?
The answer is yes; in fact, there are plenty of examples of tensorly faithful cotilting modules over non-commutative
Artinian rings. For instance, if $R$ is either a left Artinian local ring or the group ring of a finite group $G$
over a commutative Artinian ring, then any cotilting module $C_R$ is tensorly faithful; see \cite[Lemma 2.14 and the
proof of Remark 2.15]{DMT}.
\item[(iii)] Are there sufficient examples of tensorly faithful cotilting modules over non-Artinian, non-commutative
Noetherian rings? The answer is again affirmative. To demonstrate this, let $\Gamma$ be a non-Artinian Gorenstein 
commutative ring. Suppose $Q$ is a finite acyclic quiver with at least two arrows. It is straightforward to 
verify that path algebra $R=\Gamma Q$ is a non-Artinian, non-commutative Noetherian $\Gamma$-algebra; see 
\cite[Definition 2.2.5]{S}. By Example \ref{4.8}, $C_R=R$ is a tensorly faithful cotilting $R$-module. 
\item[(iv)] Can the assumption \enquote{tensorly faithful} in Theorems \ref{3.5} and \ref{4.4} be relaxed? We do not
yet know the answer to this question.
\end{itemize}
\end{remark}

\section*{Acknowledgement} The authors express their sincere gratitude to the referee for their meticulous review of
the manuscript and valuable suggestions that significantly improved the paper. They also thank Tiago Cruz for his
insightful comments and feedback on an earlier version of this work.


\end{document}